\documentclass[12pt]{amsart}
\usepackage{amssymb,amsmath,amsfonts,latexsym,amscd}
\usepackage{bm}
\newtheorem{Theorem}{\sc Theorem}[section]
\newtheorem{Lemma}[Theorem]{\sc Lemma}
\newtheorem{Proposition}[Theorem]{\sc Proposition}
\newtheorem{Corollary}[Theorem]{\sc Corollary}

\newtheorem{Example}[Theorem]{\sc Example}
\newtheorem{Remark}[Theorem]{\sc Remark}
\newtheorem{Note}[Theorem]{\sc Note}
\newtheorem{Question}{\sc Question}
\newtheorem{ass}[Theorem]{\sc Assumption}
\theoremstyle{definition}
\newtheorem{Definition}[Theorem]{\sc Definition}

\newcommand{\bt}{\begin{Theorem}}
\def\beginlem{\begin{Lemma}}
\def\beginprop{\begin{Proposition}}
\def\begincor{\begin{Corollary}}
\def\begindef{\begin{Definition}}
\def\beginexamp{\begin{Example}}
\def\beginrem{\begin{Remark}}
\def\beginq{\begin{Question}}
\def\beginass{\begin{ass}}
\def\beginnote{\begin{Note}}
\newcommand{\et}{\end{Theorem}}
\def\endlem{\end{Lemma}}
\def\endprop{\end{Proposition}}
\def\endcor{\end{Corollary}}
\def\enddef{\end{Definition}}
\def\endexamp{\end{Example}}
\def\endrem{\end{Remark}}
\def\endq{\end{Question}}
\def\endass{\end{ass}}
\def\endnote{\end{Note}}

\newcommand{\norm}[1]{\lVert #1 \rVert}
\newcommand{\ot}{\otimes}
\newcommand{\ut}{\underline{\ot}}
\newcommand\ov[1]{\overline{#1}}

\newcommand{\bdy}{\partial}

\newcommand{\clq}{\mathcal{Q}}
\newcommand{\cls}{\mathcal{S}}

\newcommand{\Nat}{\mathbb{N}}
\newcommand{\Int}{\mathbb{Z}}
\newcommand{\dsc}{\mathbb{D}}
\newcommand{\cplx}{\mathbb{C}}
\newcommand{\cle}{\mathbb{T}}

\newcommand{\z}{\bm{z}}
\newcommand{\w}{\bm{w}}

\DeclareMathOperator{\bd}{bdy}

\begin{document}

\title[Essential normality and boundary representations]
{On quotient modules of $H^2(\mathbb{D}^n)$:
essential normality and boundary representations}

\author [Das]{B. Krishna Das}
\address{Department of Mathematics, Indian Institute of Technology Bombay,
Powai, Mumbai, India- 400076} \email{dasb@math.iitb.ac.in, bata436@gmail.com}
\author[Gorai]{Sushil Gorai}
\address{Department of Mathematics and Statistics, Indian Institute
    of Science Education and Research Kolkata, Mohanpur 741 246,
West Bengal, India} \email{sushil.gorai@iiserkol.ac.in}

\author[Sarkar]{Jaydeb Sarkar}
\address{Indian Statistical Institute,
Statistics and Mathematics Unit, 8th Mile, Mysore Road, Bangalore,
560059, India} \email{jay@isibang.ac.in, jaydeb@gmail.com}




\begin{abstract}
Let $\mathbb{D}^n$ be the open unit polydisc in $\mathbb{C}^n$, $n
\geq 1$, and let $H^2(\mathbb{D}^n)$ be the Hardy space over
$\mathbb{D}^n$. For $n\ge 3$, we show that if $\theta \in H^\infty(\mathbb{D}^n)$
is an inner function, then the $n$-tuple of commuting operators
$(C_{z_1}, \ldots, C_{z_n})$ on the Beurling type quotient module
$\mathcal{Q}_{\theta}$ is not essentially normal, where
\[\mathcal{Q}_{\theta} =
H^2(\mathbb{D}^n)/ \theta H^2(\mathbb{D}^n) \quad \mbox{and} \quad
C_{z_j} = P_{\mathcal{Q}_{\theta}}
M_{z_j}|_{\mathcal{Q}_{\theta}}\quad (j = 1, \ldots, n).\] Rudin's
quotient modules of $H^2(\mathbb{D}^2)$ are also shown to be not
essentially normal. We prove several results concerning boundary
representations of $C^*$-algebras corresponding to different classes
of quotient modules including doubly commuting quotient modules and
homogeneous quotient modules.
\end{abstract}

\maketitle

\section{Introduction}

In this paper, we intend to study essential normality and boundary
representations of a class of quotient modules of the Hardy module
over the unit polydisc $\mathbb{D}^n$ in $\cplx^n$, $n > 1$. To be
more specific, let $H^2(\mathbb{D}^n)$, $n \geq 1$, denote the Hardy
space of holomorphic functions on $\dsc^n$. We also call
$H^2(\mathbb{D}^n)$ the Hardy module over $\mathbb{C}[z_1, \ldots,
z_n]$ (see Section 2 for definition). Let $(M_{z_1}, \ldots,
M_{z_n})$ denote the (commuting) $n$-tuple of multiplication
operators by the coordinate functions on $H^2(\mathbb{D}^n)$. A
closed subspace $\cls$ of $H^2(\dsc^n)$ is called a
\textit{submodule} if $M_{z_i} \cls \subseteq \cls$ for all
$i=1,\dots,n$, and a closed subspace $\clq$ of $H^2(\dsc^n)$ is a
\emph{quotient module} if $\clq^{\perp}$ $(\cong H^2(\mathbb{D}^n)/
\clq)$ is a submodule. A quotient module $\clq$ is said to be of
\textit{Beurling type} \cite{GW} if
\[
\clq = \clq_{\theta} := H^2(\mathbb{D}^n) \ominus \theta H^2(\mathbb{D}^n)
\cong H^2(\mathbb{D}^n)/ \theta H^2(\mathbb{D}^n),
\]
for some inner function $\theta \in H^\infty(\mathbb{D}^n)$ (that
is, $\theta$ is a bounded analytic function on $\mathbb{D}^n$ and
$|\theta| = 1$ a.e. on the distinguished boundary $\mathbb{T}^n$ of
$\dsc^n$). We denote by $\cls_{\theta}$ the submodule $\theta
H^2(\mathbb{D}^n)$ of $H^2(\mathbb{D}^n)$. A quotient module $\clq$
of $H^2(\dsc^n)$ is \textit{essentially normal} \cite{chg} if the
commutator $[C_{z_i}, C^{*}_{z_j}]$ is compact for all $1\le i,j\le
n$, where
\[
C_{z_i} = P_{\clq}M_{z_i}|_{\clq}\quad \quad (i=1,\dots,n).
\]

Essential normality of Hilbert modules is a much studied object in
operator theory and function theory. It establishes important
connections between operator theory, algebraic geometry, homology
theory and complex analysis through the BDF theory~\cite{bdf}.
It is well known that any proper quotient module of
$H^2(\mathbb{D})$ is of Beurling-type and essentially normal. This,
however, does not hold in general:
\begin{itemize}
\item [(1)]
For $n = 2$ a Beurling type quotient module $\clq_{\theta}
\subseteq H^2(\mathbb{D}^2)$ is essentially normal if and only if
$\theta$ is a rational inner function of degree at most $(1, 1)$
\cite{GW}.
\smallskip

\item [(2)] For $n \geq 2$, a quotient module $\clq$ is a
Beurling type quotient module of $H^2(\mathbb{D}^n)$ if and only if
$\clq^{\perp}$ is a doubly commuting submodule \cite{SSW}.
\end{itemize}
An incomplete list of references on the study of essential normality of
different classes of quotient modules, including Clark type quotient modules and homogeneous quotient modules,
over the bidisc is: ~\cite{chg}, \cite{clark},
\cite{GW_2}, \cite{GW}, \cite{GP} and \cite{PW}.

In this paper we first investigate the essential normality of
certain classes of quotient modules including Beurling-type quotient
modules of $H^2(\mathbb{D}^n)$, $n \geq 3$.  We prove that the
Beurling type quotient modules of $H^2(\mathbb{D}^n)$ $(n\ge 3)$
 and Rudin quotient modules of $H^2(\mathbb{D}^2)$
are not essentially normal. We obtain a complete
characterization for essential normality of doubly commuting
quotient modules of an analytic Hilbert module (defined in Section
2) over $\cplx[\z]$ including $H^2(\mathbb{D}^n)$ and the weighted
Bergman modules $L^2_{a,\bm{\alpha}}(\dsc^n)$
$(\bm{\alpha}\in\Int^n, \alpha_i> -1, i=1,\dots,n)$ as special cases
$(n\ge 2)$.

We also study boundary representations, in the sense of Arveson
(\cite{Arv1}, \cite{Arv2}), of the $C^*$-algebra $C^*(\clq)$ for
different classes of quotient modules $\clq$ of $H^2(\dsc^n)$. Here,
given a quotient module $\clq$, we denote by $B(\clq)$ and
$C^*(\clq)$ the Banach algebra and the $C^*$-algebra generated by
$\{I_{\clq}, C_{z_i}\}_{i=1}^n$, respectively. For convenience in
notation we put
\[
B(\clq) = B(C_{z_1},\dots,C_{z_n}), \; \mbox{and}\; C^*(\clq) =
C^*(C_{z_1},\dots,C_{z_n}).
\]
It is well known that for an essentially normal quotient module
$\clq$ of $H^2(\mathbb{D}^n)$, $B(\clq)$ is an irreducible operator
algebra and the $C^*$-algebra $C^*(\clq)$ contains all compact
operators on $\clq$ (see Proposition 2.5 in \cite{Arv3}, and Theorem
3.3 and Lemma 3.4 in \cite{dqm}).

Let us also recall the definition of the boundary representations
and some relevant results from operator algebras. Let $A$ be an
operator algebra with identity, and let $C^*(A)$ be the
$C^*$-algebra generated by $A$. An irreducible representation
$\omega$ of $C^*(A)$ is a \emph{boundary representation} relative to
$A$ if $\omega|_{A}$ has a unique completely positive (CP) extension
to $C^*(A)$. An operator algebra $A$ has trivial Shilov ideal (\cite{dk})
if
\[\bigcap_{\omega\in \bd(A)}\ker \omega=\{0\},\]
where $\bd(A)$ denotes the collection of all boundary
representations of $C^*(A)$ relative to $A$.
It is of great interest and
importance to identify operator algebras with trivial Shilov ideal.
 In the particular case of irreducible
operator algebras containing compact operators, triviality of Shilov ideal
 and the fact that the
identity representation is a boundary representation are closely
related.

\bt\textup{(}\cite[Proposition 2.1.0]{Arv2}\textup{)}\label{identity}
Let $A$ be an irreducible operator
algebra with identity, and let $C^*(A)$ contain all the compact
operators. Then $A$ has trivial Shilov ideal
if and only if the identity representation of $C^*(A)$ is a boundary
representation relative to $A$.
\et

In our context, if $\clq$ is an essentially normal quotient module of
$H^2(\dsc^n)$ then $B(\clq)$ is irreducible and $C^*(\clq)$ contains
all the compact operators on $\clq$. Therefore, it is natural to ask
whether the identity representation of $C^*(\clq)$ is a boundary representation
relative to $B(\clq)$ for the case when $\clq$ is an essentially normal quotient module
of $H^2(\dsc^n)$. This problem has a complete solution for the
case $n = 1$ (see Arveson \cite[Theorem 3.5.3]{Arv1},\cite[Corollary 1]{Arv2}):

\bt[\textsf{Arveson}] \label{for n=1} Let $\clq_{\theta}$ be a
quotient module of $H^2(\dsc)$. Then the identity representation of
$C^*(\clq_{\theta})$ is a boundary representation relative to
$B(\clq_{\theta})$ if and only if $Z_{\theta}$ is a proper subset of
$\cle$, where $Z_{\theta}$ consists of all points $\lambda$ on
$\cle$ for which $\theta$ cannot be continued analytically from
$\dsc$ to $\lambda$. \et

For the class of essentially normal Beurling type quotient modules
of $H^2(\dsc^2)$, the following characterization was obtained in
~\cite{GW}.

\bt[\textsf{Guo and Wang}] Let $\theta \in
H^\infty(\mathbb{D}^2)$ be a rational inner function of degree at
most $(1, 1)$, and $\clq_{\theta}$ be the corresponding essentially
normal quotient module of $H^2(\mathbb{D}^2)$. Then the identity
representation of $C^*(\clq_{\theta})$ is a boundary representation
relative to $B(\clq_{\theta})$ if and only if $\theta$ is not a one
variable Blaschke factor. \et

In this paper, we study the same problem
for several classes of quotient modules of some Hilbert modules over $\dsc^n$, $n\ge 2$.
To be more precise, we study boundary representations for doubly commuting quotient modules
of an analytic Hilbert module over $\cplx[\z]$, and obtain some
direct results for the case of $H^2(\mathbb{D}^n)$ and
$L^2_a(\dsc^n)$ $(n\ge 2)$ (see Corollaries~\ref{thm-dblycomm-bdyrep}
and~\ref{Bergman}). We also consider the class of
homogeneous quotient modules of
$H^2(\dsc^2)$.

The paper is organized as follows. After obtaining some preliminary results
in Section 2, we consider essential normality of Beurling type quotient module of $H^2(\dsc^n)$ $(n\ge 3)$,
doubly commuting quotients modules of an analytic Hilbert module over $\cplx[\z]$ and Rudin quotient module of $H^2(\dsc^2)$
in Section 3. Section 4 is devoted to the study of boundary representations for doubly commuting quotient modules.
In Section 5, we discuss boundary
representations for
homogeneous quotient modules of
$H^2(\dsc^2)$.

\section{Preparatory results}\label{sec-prep}

In this section we recall some definitions, and prove some
elementary results which will be used later. We begin by briefly
recalling the definition of the Hardy module.

Let $\dsc^n=\{\bm{z} = (z_1,\dots, z_n)\in\cplx^n: |z_i|\le 1,
i=1,\dots,n\}$ denote the unit polydisc in $\mathbb{C}^n$. We denote
by $\mathbb{N}$ the set of all natural numbers including $0$. Set
$\mathbb{N}^n = \{\bm{k} = (k_1, \ldots, k_n): k_j \in \mathbb{N}, j
= 1, \ldots, n\}$ and $\z^{\bm{k}}: = z_1^{k_1} \cdots z_n^{k_n}$
for all $\bm{z} \in \mathbb{C}^n$ and $\bm{k} \in \mathbb{N}^n$.
Then the \textit{Hardy space} $H^2(\mathbb{D}^n)$ over the polydisc
$\mathbb{D}^n$ is defined as the space of all holomorphic functions
$f = \sum_{\bm{k} \in \mathbb{N}^n} a_{\bm{k}} \z^{\bm{k}}$ on
$\mathbb{D}^n$ such that $\|f\|^2:= \sum_{\bm{k} \in \mathbb{N}^n}
|a_{\bm{k}}|^2 < \infty$. It is well known that $H^2(\mathbb{D}^n)$
is a reproducing kernel Hilbert space corresponding to the Szeg\"{o}
kernel
\[\mathbb{S}(\bm{z}, \bm{w}) = \prod_{i=1}^n (1 - z_i
\bar{w}_i)^{-1}, \quad \quad (\bm{z}, \bm{w} \in \mathbb{D}^n)\]and
$(M_{z_1}, \ldots, M_{z_n})$ is a commuting tuple of isometries on
$H^2(\mathbb{D}^n)$, where \[(M_{z_i} f)(\w) = w_i f(\w) \quad \quad
(f \in H^2(\mathbb{D}^n), \w \in \dsc^n, i = 1, \ldots, n).\]We
represent the $n$-tuple of multiplication operators $(M_{z_1},
\ldots, M_{z_n})$ on $H^2(\mathbb{D}^n)$ as a Hilbert module over
$\mathbb{C}[\bm{z}] : = \mathbb{C}[z_1, \ldots, z_n]$ with the
following module action:
\[
\mathbb{C}[\bm{z}] \times H^2(\mathbb{D}^n) \to H^2(\mathbb{D}^n),\quad
(p, f)\mapsto p(M_{z_1}, \ldots, M_{z_n}) f.
\]
With the above module action $H^2(\dsc^n)$ is called the
\textit{Hardy module} over $\cplx[\bm{z}]$.

We also need to recall the definition of the normalized kernel
function corresponding to the Szeg\"{o} kernel on $\mathbb{D}^n$.
For each $\bm{w} \in\dsc^n$, the \textit{normalized kernel} function
$K_{\bm{w}}$ of $H^2(\dsc^n)$
 is defined by
\[
K_{\bm{w}}(\bm{z}):= \frac{1} {\|\mathbb{S}(\cdot, \bm{w})\|}
\mathbb{S}(\bm{z}, \bm{w}) = \prod_{i=1}^n \sqrt{(1-|w_i|^2)}
\frac{1}{1-\ov{w}_i z_i} \quad \quad (\bm{z} \in \mathbb{D}^n),
\]
where $\mathbb{S}(\cdot, \bm{w})(\bm{z}) = \mathbb{S}(\bm{z},
\bm{w})$ for all $\bm{z} \in \mathbb{D}^n$. This notion is useful
when one studies the Hardy space over $\mathbb{D}^n$, $n > 1$.

\beginlem \label{weak convergence}
Let $l \in \{1, \ldots, n\}$ be a fixed integer, and let $\bm{w}_l =
(w_1,\dots, w_{l-1}, w_{l+1},\dots, w_n)$ be a fixed point in
$\dsc^{n-1}$. Then $K_{(\bm{w}_l, w)}$  converges weakly to $0$ as
$w$ approaches to $\partial \mathbb{D}$, where $(\bm{w}_l,
w) = (w_1,\dots, w_{l-1}, w, w_{l+1},\dots, w_n)$.
\endlem
\begin{proof} For each $p\in \cplx[\bm{z}]$,
\begin{align}
\langle K_{(\bm{w}_l, w)}, p\rangle &= \overline{p(\bm{w}_l, w)}
\sqrt{1-|w|^2} \prod_{i=1, i\neq l}^n\sqrt{1-|w_i|^2} ,
\label{eq-pol}
\end{align}
which converges to zero as $w$ approaches to $\partial \mathbb{D}$.
For an arbitrary $f\in H^2(\dsc^n)$, the result now follows from the
fact that $\norm{K_{\bm{\lambda}}}=1$ for all $\bm{\lambda}
\in\dsc^n$ and $\cplx[\bm{z}]$ is dense in $H^2(\dsc^n)$.
\end{proof}

For a closed subspace $\cls$ of a Hilbert space $\mathcal{H}$, the
orthogonal projection of $\mathcal{H}$ onto $\cls$ is denoted
by $P_{\cls}$. For an inner function $\theta \in H^\infty(\mathbb{D}^n)$,
it is well known that
\[P_{\cls_{\theta}} = M_{\theta} M_{\theta}^*\ \text{ and }\
P_{\clq_{\theta}} = I_{H^2(\mathbb{D}^n)} - M_{\theta}
M_{\theta}^*,\]where $M_{\theta}$ is the multiplication operator
defined by
\[
(M_{\theta} f) (\bm{w}) = \theta(\bm{w}) f(\bm{w}) \quad \quad
(\bm{w} \in \mathbb{D}^n, f \in H^2(\mathbb{D}^n)).
\]
It follows from the reproducing property of the Szeg\"{o} kernel
that
\[M_{\theta}^* K(\cdot, \bm{w}) = \overline{\theta(\bm{w})} K(\cdot,
\bm{w}),\]where $K(\cdot, \bm{w}) := K_{\bm{w}}$, $\bm{w} \in
\mathbb{D}^n$. In particular, one has
\[
P_{\cls_{\theta}}(K_{\bm{w}}) = M_{\theta}M_{\theta}^* K_{\bm{w}} =
\overline{\theta(\bm{w})}\theta K_{\bm{w}} \quad \quad (\bm{w} \in
\mathbb{D}^n).
\]
These observations yield the following lemma.

\beginlem
Let ${\theta}$ be an inner function in $H^\infty(\dsc^n)$. Then
\begin{equation}
\label{projection formula} P_{\clq_{\theta}}(K_{\bm{w}})=
(1-\ov{\theta(\bm{w})}\theta) K_{\bm{w}} \quad \quad (\bm{w} \in
\mathbb{D}^n).
\end{equation}
\endlem

We now recall the definition of an analytic Hilbert module over $\cplx[z]$
(see \cite{dqm}). Let $k : \mathbb{D} \times \mathbb{D} \rightarrow
\mathbb{C}$ be a positive definite function such that $k(z, w)$ is
analytic in $z$ and anti-analytic in $w$. Let $\mathcal{H}_k \subseteq
\mathcal{O}(\dsc, \mathbb{C})$ be the corresponding reproducing
kernel Hilbert space, where $\mathcal{O}(\dsc, \mathbb{C})$ denotes the
set of all holomorphic functions on the unit disc. The Hilbert space $\mathcal{H}_k$ is said to
be a \textit{reproducing kernel Hilbert module} over $\mathbb{C}[z]$
if the multiplication operator $M_z$ is bounded on $\mathcal{H}_k$.

\begin{Definition}\label{clasRKHM}
A reproducing kernel Hilbert module $\mathcal{H}_{k}$ over
$\mathbb{C}[z]$ is said to be an analytic Hilbert module over
$\mathbb{C}[z]$ if $k^{-1}(z, w)$ is a polynomial in $z$ and
$\bar{w}$.
\end{Definition}

Typical examples of analytic Hilbert modules are the Hardy module
$H^2(\mathbb{D})$ with Szeg\"{o} kernel
\[K(z,w)=\frac{1}{1-z\bar{w}}\quad (z,w\in\dsc)\] and the weighted Bergman modules
$L^2_{a,\alpha}(\dsc)$ $(\alpha >-1, \alpha\in\Int)$ with kernel
\[K_{a,\alpha}(z,w)=\frac{1}{(1-z\bar{w})^{\alpha+2}} \quad (z,w\in\dsc, \alpha> -1). \]It is known
that a quotient module of an analytic Hilbert module is irreducible,
that is, $C_z$ does not have any non-trivial reducing subspace (cf.
Theorem 3.3 and Lemma 3.4 in \cite{dqm}). Using this, we obtain the
next lemma.

\begin{Lemma}
\label{commutative} Let $\clq$ be a non-zero quotient module of an
analytic Hilbert module  $\mathcal{H}$ over $\cplx[z]$. Then $[C_z,
C_z^*]=0$ if and only if $\clq$ is one dimensional.
\end{Lemma}

\begin{proof}
First note that for any non-zero quotient module $\clq$ of $\mathcal{H}$, the
$C^*$-algebra $C^*(\clq)$ is irreducible. If $C_z$ is normal, then
$C^*(\clq)\subseteq C^*(\clq)'=\cplx I$. Thus $C^*(\clq)=\cplx
I$, and therefore, $\clq$ is one dimensional. The converse part is
trivial, and the proof follows.
\end{proof}

Let $\{k_i\}_{i=1}^n$ be positive definite functions on
$\mathbb{D}\times\mathbb{D}$. Then $\mathcal{H}_K :=
\mathcal{H}_{k_1}\ot\cdots\ot\mathcal{H}_{k_n}$ is said to be an
\emph{analytic Hilbert module over $\mathbb{C}[\z]$} if
$\mathcal{H}_{k_i}$ is an analytic Hilbert module over $\cplx[z]$
for all $i=1,\dots,n$. In this case, $\mathcal{H}_K \subseteq
\mathcal{O}(\dsc^n,\mathbb{C})$ and
\[
K(\z, \w) = \prod_{i=1}^n k_i(z_i, w_i) \quad \quad (\z, \w \in
\dsc^n),
\]
is the reproducing kernel function of $\mathcal{H}_K$ (cf.
\cite{dqm}). In the sequel, we will often identify $M_{z_i}$ on $\mathcal{H}_K$
with the operator $I_{\mathcal{H}_{k_1}} \otimes \cdots \otimes
\underbrace{M_z}\limits_{\textup{i-th place}} \otimes \cdots \otimes
I_{\mathcal{H}_{k_n}}$, $i = 1, \ldots, n$, on the $n$-fold Hilbert
space tensor product
$\mathcal{H}_{k_1}\ot\cdots\ot\mathcal{H}_{k_n}$. We end this section with a result on essential
normality of a Beurling type quotient module $\clq_{\theta}$, where
$\theta$ is a one variable inner function in $\mathbb{D}^n$.

\beginlem \label{lem-one-var-q-mod}
Let $\theta \in H^\infty(\mathbb{D}^n)$ be a one variable inner
function and $n \geq 3$. Then $\clq_{\theta}$ is not essentially
normal.
\endlem

\begin{proof}
Without loss of generality we may assume that
$\theta(\bm{z})=\theta'(z_1)$ for some inner function $\theta' \in
H^\infty(\dsc)$. Then it follows that $\cls_{\theta} =
\cls_{\theta'} \otimes H^2(\mathbb{D}^{n-1})$ and
\[\clq_{\theta}= H^2(\dsc^n)\ominus \theta H^2(\dsc^n)
= \clq_{\theta'}\otimes H^2(\dsc^{n-1}).\]
Now we compute the self commutator of $C_{z_2}$:
\begin{align*}
[C_{z_2}, C_{z_2}^*]&=P_{\clq_{\theta}}
M_{z_2}M_{z_2}^*|_{\clq_{\theta}} - P_{\clq_{\theta}}
M_{z_2}^*P_{\clq_{\theta}} M_{z_2}|_{\clq_{\theta}}\notag\\
&= P_{\clq_{\theta}} M_{z_2}M_{z_2}^*|_{\clq_{\theta}} -
I_{\clq_{\theta}} +  P_{\clq_{\theta}} M_{z_2}^*P_{\cls_{\theta}}
M_{z_2}|_{\clq_{\theta}}. \label{eq-2ndselfcomm}
\end{align*}
Using the fact \[P_{\cls_{\theta}} M_{z_2}|_{\clq_{\theta'} \otimes \mathbb{C}
\otimes H^2(\mathbb{D}^{n-2})} = (P_{\cls_{\theta'}} \otimes
I_{H^2(\mathbb{D})} \otimes I_{H^2(\mathbb{D}^{n-2})})
M_{z_2}|_{\clq_{\theta'} \otimes \mathbb{C} \otimes
H^2(\mathbb{D}^{n-2})} = 0,\]and \[M_{z_2}^*|_{\clq_{\theta'} \otimes \mathbb{C}
\otimes H^2(\mathbb{D}^{n-2})} = 0,\]
we conclude that
\[
[C_{z_2}, C_{z_2}^*]|_{\clq_{\theta'} \otimes \mathbb{C} \otimes
H^2(\mathbb{D}^{n-2})} = - I_{\clq_{\theta}}|_{\clq_{\theta'}
\otimes \mathbb{C} \otimes
H^2(\mathbb{D}^{n-2})} = -I_{\clq_{\theta'}\otimes \mathbb{C} \otimes
H^2(\dsc^{n-2})}.
\]
Since $n\ge 3$, $[C_{z_2}, C_{z_2}^*]|_{\clq_{\theta'} \otimes \mathbb{C} \otimes
H^2(\mathbb{D}^{n-2})}$ is not compact, and hence the
commutator $[C_{z_2}, C_{z_2}^*]$ is not compact.
This completes the proof.
\end{proof}

\section{Essential normality}\label{sec-essential-normality}

Our purpose in this section is to prove
a list of results concerning essential normality for certain classes of
quotient modules.  We begin with the class of Beurling type
quotient modules of $H^2(\dsc^n)$, $n \geq 3$.

\bt\label{T:nonessnormal} Let $\theta$ be an inner function in $H^\infty(\mathbb{D}^n)$ and $n \geq 3$. Then
$\clq_{\theta}$ is not essentially normal.
\et
\begin{proof} First recall that the multiplication tuple $(M_{z_1},\dots,M_{z_n})$ on $H^2(\dsc^n)$ is doubly commuting, that is
\[
 M_{z_i}^*M_{z_j}=M_{z_j}M_{z_i}^*
\]
if $i\neq j$. Now by Lemma \ref{lem-one-var-q-mod},
we may assume without loss of generality that $\theta$
depends on both $z_1$ and $z_2$
variables. We show that $[C_{z_1},C_{z_2}^*]$ is not compact. To
this end, we compute
\begin{align} [C_{z_1},C_{z_2}^*]&= P_{\clq_{\theta}}
M_{z_1}M_{z_2}^*|_{\clq_{\theta}} - P_{\clq_{\theta}}
M_{z_2}^*P_{\clq_{\theta}} M_{z_1}|_{\clq_{\theta}} \notag
= P_{\clq_{\theta}} M_{z_2}^*P_{\cls_{\theta}} M_{z_1}|_{\clq_{\theta}} \notag\\
&=P_{\clq_{\theta}} M_{z_2}^*P_{\cls_{\theta}\ominus
(z_1\cls_{\theta} +z_2\cls_{\theta})} M_{z_1}|_{\clq_{\theta}}\notag
+ P_{\clq_{\theta}} M_{z_2}^*P_{z_1\cls_{\theta} +z_2\cls_{\theta}}
M_{z_1}|_{\clq_{\theta}}\notag\label{eq-mixcomm}.
\end{align}
Since $M_{z_1}$ and $M_{z_2}$ are isometries, we have
\[
P_{\clq_{\theta}}M_{z_i}^*P_{z_i\cls_{\theta}}=0 \quad \quad (i=1,
2).
\]
This implies
\[P_{\clq_{\theta}} M_{z_2}^*P_{z_1\cls_{\theta} +z_2\cls_{\theta}}
M_{z_1}|_{\clq_{\theta}} = 0,\]and
\[[C_{z_1},C_{z_2}^*] =
P_{\clq_{\theta}} M_{z_2}^*P_{\cls_{\theta}\ominus (z_1\cls_{\theta}
+z_2\cls_{\theta})}
 M_{z_1}|_{\clq_{\theta}}.\]
On the other hand, since $\cls_{\theta} = \theta H^2(\mathbb{D}^n)$,
we have
\[
\begin{split}
\cls_{\theta}\ominus(z_1\cls_{\theta}+z_2\cls_{\theta}) & = \theta
H^2(\mathbb{D}^n) \ominus \theta(z_1 H^2(\mathbb{D}^n) + z_2
H^2(\mathbb{D}^n))
\\
& = \theta(\mathbb{C} \otimes \mathbb{C} \otimes
H^2(\mathbb{D}^{n-2})).
\end{split}
\]
Then for $f\in \mathbb{C} \otimes \mathbb{C} \otimes
H^2(\mathbb{D}^{n-2})$ and $g\in H^2(\dsc^n)$,
\[
 \langle M_{z_2}^* \theta f, \theta g\rangle= \langle f, z_2g\rangle=0,
\]
and therefore
\[
M_{z_2}^*(\cls_{\theta}\ominus(z_1\cls_{\theta}+z_2\cls_{\theta}))
\subseteq \clq_{\theta}.
\]
Consequently, \[[C_{z_1},C_{z_2}^*] =
M_{z_2}^*P_{\cls_{\theta}\ominus (z_1\cls_{\theta}
+z_2\cls_{\theta})} M_{z_1}|_{\clq_{\theta}}.\] By Lemma \ref{weak
convergence}, it is enough to show that $\langle
[C_{z_1},C_{z_2}^*]K_{\w}, K_{\w}\rangle$ does not converge to $0$
as $w_j$ approaches to $\bdy\dsc$ for some fixed $3\le j\le n$, and
keeping all other co-ordinates of $\w = (w_1, \ldots w_{j-1}, w_j,
w_{j+1}, \ldots, w_n) \in\dsc^n$ fixed. To this end, let $\w
\in\dsc^n$. Since $\{\theta z_3^{m_3}\cdots z_n^{m_n}:
m_3,\dots,m_n\in \mathbb{N}\}$ is an orthonormal basis of
$\cls_{\theta}\ominus(z_1\cls_{\theta}+z_2\cls_{\theta})$, we have
\[\begin{split}
P_{\cls_{\theta}\ominus(z_1\cls_{\theta}+z_2\cls_{\theta})}(z_2K_{\w})&=
\sum_{m_3,\dots,m_n\in \mathbb{N}} \langle z_2 K_{\w}, \theta
z_3^{m_3}\cdots z_n^{m_n}\rangle
\theta z_3^{m_3}\cdots z_n^{m_n} \notag\\
&= \sum_{m_3,\dots,m_n\in\mathbb{N}} \langle K_{\w}, z_3^{m_3}\cdots
z_n^{m_n}(M_{z_2}^*\theta)\rangle  \theta z_3^{m_3}\cdots z_n^{m_n}\notag\\
&= \frac{1}{\|\mathbb{S}(\cdot, \w)\|} \theta \sum_{m_3,\dots,m_n\in
\mathbb{N}}(\ov{w_3}z_3)^{m_3}\dots
(\ov{w_n}z_n)^{m_n} \ov{M_{z_2}^* \theta (\w)}\notag\\
&= \ov{M_{z_2}^*\theta(\w)} \prod_{j=1}^2 (1-|w_j|^2)^{\frac{1}{2}}
\Big(\prod_{i=3}^n K_{w_i}\Big) \theta. \end{split}\]
Here $K_{w_i}=K_{\w}$ with $\w=(0,\dots,w_i,\dots,0)$.
Thus
\[\begin{split}
\langle [C_{z_1}, C_{z_2}^*]K_{\w}, K_{\w}\rangle &= \langle
M_{z_2}^* P_{\cls_{\theta} \ominus (z_1\cls_\theta
+z_2\cls_{\theta})} M_{z_1}P_{\clq_\theta} K_{\w}, K_{\w}\rangle
\notag\\&= \langle M_{z_1}P_{\clq_{\theta}} K_{\w},
P_{\cls_{\theta} \ominus (z_1 \cls_{\theta} + z_2 \cls_{\theta})}
(z_2 K_{\w})\rangle \notag\\
&=(M_{z_2}^* \theta)(\w) \prod_{j=1}^2
(1-|w_j|^2)^{\frac{1}{2}}\Big\langle M_{z_1} P_{\clq_{\theta}}
K_{\w}, \prod_{i=3}^n K_{w_i} \theta \Big\rangle\notag\\
&=(M_{z_2}^*
\theta)(\w)\prod_{j=1}^2(1-|w_j|^2)^{\frac{1}{2}}\Big\langle
M_{z_1}(1 - \ov{\theta(\w)} \theta) K_{\w}, \prod_{i=3}^n K_{w_i}
\theta \Big \rangle,
\end{split}\]
where the last equality follows
from \eqref{projection formula}. Since $M_{z_1}^* (\prod_{i=3}^n
K_{w_i}) = 0$ and $M^*_\theta M_\theta = I_{H^2(\mathbb{D}^n)}$, we
have
\[\langle M_{z_1} \theta K_{\w}, \prod_{i=3}^n K_{w_i} \theta
\Big \rangle = \langle \theta  M_{z_1}  K_{\w}, \prod_{i=3}^n
K_{w_i} \theta \Big \rangle = \langle K_{\w},
M_{z_1}^*(\prod_{i=3}^n K_{w_i}) \Big \rangle = 0.\]
Therefore,
\[\begin{split} \langle [C_{z_1}, C_{z_2}^*]K_{\w}, K_{\w}\rangle &
=(M_{z_2}^*\theta)(\w)\prod_{j=1}^2
(1-|w_j|^2)^{\frac{1}{2}} \Big\langle M_{z_1} K_{\w}, \prod_{i=3}^n
K_{w_i} \theta \Big\rangle\\ &
=(M_{z_2}^*\theta)(\w)\prod_{j=1}^2 (1-|w_j|^2)^{\frac{1}{2}}
\Big\langle K_{\w}, \prod_{i=3}^n K_{w_i} (M_{z_1}^* \theta)
\Big\rangle
\\ &=(M_{z_2}^*\theta)(\w) \prod_{j=1}^2(1-|w_j|^2)^{\frac{1}{2}} \;
\Big(\ov{M_{z_1}^*\theta(\w)} \frac{1}{\|\mathbb{S}(\cdot,
\w)\|} \prod_{i=3}^n \frac{1}{(1-|w_j|^2)^{\frac{1}{2}}}\Big)\\ & =
(M_{z_2}^*\theta)(\w) \; \ov{(M_{z_1}^*\theta)(\w)}
\prod_{j=1}^2(1-|w_j|^2).
\end{split}\]
Since $\theta$ depends on both $z_1$ and $z_2$ variables, $M_{z_1}^*\theta$
and $M_{z_2}^*\theta$ are non-zero functions.
Therefore it follows that there exist an $l
\in \{3, \ldots, n\}$ and $\w_k=(w_1,\dots, w_{l-1}, \lambda_k, w_{l+1},\dots,w_n)\in\dsc^n$ $(k\in\Nat)$, where
$\{\lambda_k\}\to \lambda\in\bdy\dsc$ and $w_i'$s are fixed, such that the limit of
\[(M_{z_2}^*\theta)(\w_k) \; \ov{(M_{z_1}^*\theta)(\w_k)}
\prod_{j=1}^2(1-|w_j|^2)\]
as $k\to \infty $ is a non-zero number.
This completes the proof.
\end{proof}

We now proceed to the case of doubly commuting quotient modules
of an analytic Hilbert module over $\mathbb{C}[\z]$.
Let $\clq$ be a quotient module of an analytic Hilbert module
$\mathcal{H}_K$ over $\mathbb{C}[\z]$. It is known that $\clq$ is
doubly commuting (that is, $[C_{z_i}, C_{z_j}^*] = 0$ for all $1
\leq i < j \leq n$) if and only if $\clq=\clq_1\ot\cdots\ot\clq_n$
for some quotient module $\clq_i$ of $\mathcal{H}_{k_i}$,
$i=1,\dots,n$ (see \cite{dqm},
\cite{bshift} and \cite{sarkar}).


\bt\label{thm-dblycomm-essnormal} Let
$\clq=\clq_1\ot\cdots\ot\clq_n$ be a doubly commuting quotient
module of an analytic Hilbert module $\mathcal{H}_K
=\mathcal{H}_{k_1}\ot\cdots\ot\mathcal{H}_{k_n}$ over
$\mathbb{C}[\z]$, $n\ge 2$. Then $\clq$ is essentially normal if and
only if one of the following holds:
\begin{itemize}
\item[(i)] $\clq$ is finite dimensional.
\item[(ii)] There exits an $i \in \{1, \ldots, n\}$
such that $\clq_{i}$ is an infinite dimensional essentially normal
quotient module of $\mathcal{H}_{k_i}$, and $\clq_j \cong \mathbb{C}$
for all $j \neq i$.
\end{itemize}
\et
\begin{proof}
Let $\clq=\clq_1\ot\cdots\ot\clq_n$ be an infinite dimensional
essentially normal quotient module. Then at least one of
$\clq_{1},\dots,\clq_{n}$ is infinite dimensional. Without loss of
 generality we assume that $\clq_n$ is infinite dimensional. For
each $i=1,\dots,n$, we now compute the self-commutator:
\begin{align}
[C_{z_i},C_{z_i}^*]&= P_{\clq}M_{z_i}M_{z_i}^*|_{\clq}-
P_{\clq}M_{z_i}^*P_{\clq}M_{z_i}|_{\clq}\nonumber\\
&=P_{\clq_1}\ot\cdots\ot P_{\clq_{i-1}}\ot
\underbrace{[C_{z},C_z^*]_i}\limits_{\textup{i-th place}}\ot
P_{\clq_{i+1}}\ot\cdots\ot P_{\clq_n}, \label{3}
\end{align}
where $[C_z, C_z^*]_i$ is the self-commutator corresponding to the quotient module
$\clq_i$. Since $\clq_n$ is infinite dimensional, the compactness of
$[C_{z_i},C_{z_i}^*]$ implies that  $[C_z,C_z^*]_i=0$ for all
$i=1,\dots,n-1$. Therefore, by Lemma~\ref{commutative}, it follows that
$\clq_i \cong \mathbb{C}$, $i=1,\dots,n-1$.

\noindent Finally, for $i=n$, the compactness of $[C_{z_n},
C_{z_n}^*]=P_{\clq_1}\ot\cdots\ot P_{\clq_{n-1}}\ot [C_z, C_z^*]_n$
implies that $[C_z, C_z^*]_n$ is compact, that is, $\clq_n$ is
essentially normal.
\smallskip

For the converse, it is enough to show
that (ii) implies $\clq$ is essentially normal. Again, without loss
of generality, we assume that $\clq_n$ is infinite dimensional
essentially normal quotient module. Then it readily follows from
~\eqref{3} that $[C_{z_i}, C_{z_i}^*]=0$, $i=1,\dots,n-1$, and
$[C_{z_n}, C_{z_n}^*]$ is compact. Now the proof follows from
Fuglede-Putnam theorem.
\end{proof}
\smallskip

The above result applies, in particular, if $\mathcal{H}_K$ is
$H^2(\dsc^n)$ or the weighted Bergman modules
$L^2_{a,\bm{\alpha}}(\dsc^n)$ $(\bm{\alpha}\in \Int^n, \alpha_i>-1,
i=1,\dots,n)$. Moreover, since every quotient module of
$H^2(\mathbb{D})$ is essentially normal, by Theorem
\ref{thm-dblycomm-essnormal} we have the following corollary.

\begin{Corollary}
Let $\clq=\clq_{1}\otimes\cdots\otimes\clq_{n}$ be a doubly
commuting quotient module of $H^2(\dsc^n)$, $n \geq 2$. Then $\clq$
is essentially normal if and only if one of the following holds:
\begin{itemize}
\item[(i)] $\clq$ is finite dimensional.

\item[(ii)]  There exits an $i \in \{1, \ldots, n\}$ such that $\clq_{i}$
is infinite dimensional, and $\clq_j \cong \mathbb{C}$ for
all $j \neq i$.
\end{itemize}
\end{Corollary}

It is also well known that a quotient module $\clq$ of the Bergman module
$L^2_a(\dsc)$ is essentially normal if and only if
\[
 \mbox{dim}(\cls\ominus z\cls) <\infty,
\]
 where $\cls:=L^2_a(\dsc)\ominus \clq$ is the corresponding
 submodule (see~\cite[Theorem 3.1]{zhu}).
 Using this and Theorem~\ref{thm-dblycomm-essnormal}, we have the
 following result.
 \begin{Corollary}
 Let $\clq=\clq_{1}\otimes\cdots\otimes\clq_{n}$ be a doubly
commuting quotient module of $L_a^2(\dsc^n)$, $n \geq 2$. Then $\clq$
is essentially normal if and only if one of the following holds:
\begin{itemize}
\item[(i)] $\clq$ is finite dimensional.

\item[(ii)]  There exists an $i \in \{1, \ldots, n\}$ such that $\clq_{i}$
is infinite dimensional with
$\mbox{dim}(\cls_i\ominus z\cls_i) <\infty$ and $\clq_j \cong \mathbb{C}$ for
all $j \neq i$, where $\cls_i=L^2_a(\dsc)\ominus \clq_i$.
\end{itemize}
 \end{Corollary}

We now restrict our attention to $H^2(\dsc^2)$, and formulate
the definition of the Rudin quotient module of
$H^2(\mathbb{D}^2)$ (see \cite{CDS14}, \cite{DS}). Let $\Psi = \{\psi_n\}_{n=0}^\infty \subseteq
H^2(\mathbb{D})$ be an increasing sequence of finite Blaschke
products and $\Phi = \{\varphi_n\}_{n=0}^\infty \subseteq
H^2(\mathbb{D})$ be a decreasing sequence of Blaschke products, that
is, $\psi_{n+1}/\psi_n$ and $\varphi_{n}/\varphi_{n+1}$ are
non-constant inner functions for all $n\in\Nat$. Then the
\textit{Rudin quotient module} corresponding to $\Psi$ and $\Phi$
is denoted by $\clq_{\Psi, \Phi}$,
and defined by
\[\clq_{\Psi, \Phi}:= \mathop{\bigvee}_{n=0}^\infty
\big(\clq_{\psi_n} \otimes \clq_{\varphi_n}\big).
\]
We denote by $\cls_{\Psi,\Phi}$ the submodule $H^2(\dsc^2)\ominus\clq_{\Psi,\Phi}$
corresponding to $\clq_{\Psi,\Phi}$. The
following representations of $\clq_{\Psi,\Phi}$ and
$\cls_{\Psi,\Phi}$ are very useful:
\begin{equation}
\clq_{\Psi,\Phi}=\bigoplus_{n\ge 0}
(\clq_{\psi_n}\ominus\clq_{\psi_{n-1}})\ot \clq_{\varphi_n}\ \text{
and } \cls_{\Psi,\Phi}=\clq'\ot H^2(\dsc)\bigoplus_{n\ge 0}
(\clq_{\psi_n}\ominus\clq_{\psi_{n-1}})\ot \cls_{\varphi_n},
\label{rudin representation}
\end{equation}
where $\clq_{\psi_{-1}}:=\{0\}$ and $\clq'=H^2(\dsc)\ominus
\vee_{n\ge 0} \clq_{\psi_n}$.
The first equality follows from the fact that $\clq_{\psi_n}\subseteq \clq_{\psi_{n+1}}$
and $\clq_{\varphi_n}\supseteq \clq_{\varphi_{n+1}}$ $(n\ge 0)$ and the second equality
can be checked easily using the equality $\clq_{\Psi,\Phi}\oplus \cls_{\Psi,\Phi}=H^2(\dsc^2)$.

Next we show that the Rudin quotient modules are not essentially
normal.
\begin{Theorem}
Let $\clq_{\Psi,\Phi}$ be a Rudin quotient module of $H^2(\dsc^2)$
corresponding to an increasing sequence of
finite Blaschke products $\Psi=\{\psi_n\}_{n\ge 0}$ and a decreasing
sequence of Blaschke products $\Phi=\{\varphi_n\}_{n\ge 0}$. Then
$\clq_{\Psi,\Phi}$ is not essentially normal.
\end{Theorem}

\begin{proof}
Let $b_{\beta}$, the Blaschke factor corresponding to
$\beta\in\dsc$, be a factor of $\psi_{m+1}/\psi_{m}$ for some $m\ge
0$. For contradiction, we assume that $\clq_{\Psi,\Phi}$ is
essentially normal. Since $\clq:= \clq_{\Psi,\Phi}$ is essentially normal, for a polynomial $p$ it is easy to verify using a
simple commutator manipulation that $[C_{p(z_1)}, C_{p(z_1)}^*]$ is compact, where
$C_{p(z_1)}= P_{\clq}M_{p(z_1)}|_{\clq}$.
Now as $\psi_m$ is a finite Blaschke product and can be approximated by polynomials,
$[C_{\psi_m(z_1)}, C_{\psi_m(z_1)}^*]$ is also compact,
where
$C_{\psi_m(z_1)} = P_{\clq} M_{\psi_m(z_1)}|_{\clq}$.
Now setting $\cls:=\cls_{\Psi,\Phi}$, we
have
\begin{align}
[C_{\psi_m(z_1)}, C_{\psi_m(z_1)}^*]&= P_{\clq} M_{\psi_m(z_1)}
M_{\psi_m(z_1)}^*|_{\clq}-
P_{\clq}M_{\psi_m(z_1)}^* P_{\clq} M_{\psi_m(z_1)}|_{\clq}\notag\\
&=-P_{\clq}(I-M_{\psi_m(z_1)} M_{\psi_m(z_1)}^*)|_{\clq}+
P_{\clq}M_{\psi_m(z_1)}^* P_{\cls}M_{\psi_m(z_1)}|_{\clq}\notag\\
&=-P_{\clq}(P_{\clq_{\psi_m}}\ot I)|_{\clq}+
P_{\clq}M_{\psi_m(z_1)}^*P_{\cls}M_{\psi_m(z_1)}|_{\clq}. \label{6}
\end{align}
Since $\varphi_{m+1}$ is an infinite Blaschke product, there exists a
sequence $(\lambda_i)_{i\in\Nat}$ in the zero set of $\varphi_{m+1}$ such that
$K_{\lambda_i}\in\clq_{\varphi_{m+1}}$ and $\lambda_i$ approaches to
$\bdy\dsc$ as $i\to \infty$.
Furthermore, since $K_{\beta}\ot
K_{\lambda_i}\in \clq_{\psi_{m+1}}\ot \clq_{\varphi_{m+1}}\subseteq \clq$ and  $\psi_mK_{\beta}\ot K_{\lambda_i}\in
(\clq_{\psi_{m+1}}\ominus \clq_{\psi_{m}})\ot \clq_{\varphi_{m+1}}$ by the divisibility of $\psi_{m+1}$ and $\psi_m$,
we have $P_{\cls}(\psi_m K_{\beta}\ot K_{\lambda_i})=0$, $i\in\Nat$ (by
~\eqref{rudin representation}). Thus
\[
P_{\clq}M_{\psi_m(z_1)}^*P_{\cls}M_{\psi_m(z_1)}(K_{\beta}\ot
K_{\lambda_i})=0 \quad (i\in \Nat).
\]
Finally, from ~\eqref{6}, we have
\begin{align*}
\langle [C_{\psi_m(z_1)}, C_{\psi_m(z_1)}^*](K_{\beta}\ot
K_{\lambda_i}), K_{\beta}\ot K_{\lambda_i}\rangle &= - \langle
(P_{\clq_{\psi_m}} K_{\beta})\ot K_{\lambda_i}, K_{\beta}\ot
K_{\lambda_i}\rangle\\
&=-\langle (1-\ov{\psi_m(\beta)}\psi_m)K_{\beta},
K_{\beta}\rangle\\
&=-(1-|\psi_m(\beta)|^2),
\end{align*}
which does not converges to $0$ as $\lambda_i$ approaches to
$\bdy\dsc$.
Thus by Lemma~\ref{weak convergence}, we have the desired contradiction.
This completes the proof.
\end{proof}

\begin{Remark}
Let $m>1$. For a decreasing sequence of Blaschke products
$\{\varphi_n\}_{n=1}^m$ and an increasing sequence of finite Blaschke products
$\{\psi_n\}_{n=1}^m$, we consider the quotient module
\[
\clq=\bigvee_{n=1}^m\clq_{\psi_n}\ot\clq_{\varphi_n}.
\]
Adapting the techniques in the proof of the above theorem, one can conclude that
$\clq$ is essentially normal if and only if $\varphi_n$ is a finite
Blaschke product for all $n=1,\dots,m$. In other words, $\clq$ is
essentially normal if and only if $\clq$ is finite dimensional.
\end{Remark}

\section{Boundary Representations for doubly commuting quotient modules}\label{sec-bdyrepn}

In this section we study boundary representations for doubly
commuting quotient modules of an analytic Hilbert module over
$\cplx[\z]$. First, we prove a general result in the setting of
minimal tensor products of $C^*$-algebras. Before that we fix some
notations. We denote by $V_1\ut V_2$ the algebraic tensor product of
two vector spaces $V_1$ and $V_2$, and by $A_1 \otimes A_2$ the
minimal tensor product of two $C^*$-algebras $A_1$ and $A_2$ where
the norm on $A_1\otimes A_2$ is obtained via the identification
$A_1\otimes A_2\subseteq \mathcal{B}(\mathcal{H})\otimes \mathcal{B}(\mathcal{K})$
corresponding to any faithful representations of $A_1$ and $A_2$ in $ \mathcal{B}(\mathcal{H})$
and $\mathcal{B}(\mathcal{K})$ respectively.

The base case $(n=2)$ of the following result is due to Hopenwasser (see Lemmas 1 and 3 in
\cite{Hop}). The proof for the
general case $n$ can be obtain easily by applying Hopenwasser's result $n-1$ times and therefore we omit the proof.


\beginlem[cf. ~\cite{Hop}]
 \label{general result}
Let $A_i$ be a unital subalgebra of
$\mathcal{B}(\mathcal{H}_i)$ for some Hilbert space $\mathcal{H}_i$,
and let $C^*(A_i)$ be the irreducible $C^*$-algebra generated by
$A_i$ in $\mathcal{B}(\mathcal{H}_i)$, $i = 1,\dots,n$. Set
$A:=\ov{(A_1\underline{\ot}\cdots\underline{\ot}A_n)}$, the norm closure of
$A_1\underline{\ot}\cdots \underline{\ot} A_n$ in
$\mathcal{B}(\mathcal{H}_1\ot\cdots\ot\mathcal{H}_n)$. Then the following are
equivalent.
\begin{itemize}
\item[(i)]
 The identity representation of $C^*(A_1)\ot\cdots\ot C^*(A_n)$
 is a boundary representation relative to $A$.
\item[(ii)]
The identity representation of $C^*(A_i)$ is a boundary representation
relative to $A_i$ for all $i=1,\dots,n$.
\end{itemize}
\endlem


As a straightforward consequence of the above lemma we obtain
the following:

\bt \label{thm-analhm-bdyrep} Let $\clq=\clq_{1}\otimes \dots\otimes
\clq_{n}$ be a doubly commuting quotient module of an analytic
Hilbert module
$\mathcal{H}=\mathcal{H}_{K_1}\ot\cdots\ot\mathcal{H}_{K_n}$ over
$\cplx[\z]$. Then the following are equivalent.
\begin{itemize}
\item[(i)] The identity representation of
$C^*(\clq)$ is a boundary representation relative to
$B(\mathcal{Q})$.

\item[(ii)] The identity representation of
$C^*(\clq_i)$ is a boundary representation relative to
$B(\mathcal{Q}_i)$ for all $i=1,\dots, n$.
\end{itemize}
\et
\begin{proof}
The result follows from Lemma~\ref{general result} and the fact that
\[
C^*(\clq)=C^*(\clq_1)\ot\cdots\ot C^*(\clq_n),\] and\[
B(\clq)=\overline{B(\clq_1)\ut\cdots\ut B(\clq_n)},
\]
where the closure is in the norm topology of $B(\clq)$.
\end{proof}

The following result is now an immediate consequence of
Theorems \ref{for n=1} and \ref{thm-analhm-bdyrep}.

\begincor\label{thm-dblycomm-bdyrep} Let $\clq = \clq_{\theta_1}\otimes
\dots\otimes \clq_{\theta_n}$ be a doubly commuting quotient module
of $H^2(\dsc^n)$, where $\theta_i$, $i=1,\dots,n$, is a one variable
inner function. Then the following are equivalent.
\begin{itemize}
\item[(i)] The identity representation of
${C}^*(\clq)$ is a boundary representation relative to $B(\clq)$.
\item[(ii)] The identity representation of ${C}^*(\clq_{\theta_i})$ is a boundary
representation relative to $B(\clq_{\theta_i})$ for all $i=1,\dots,
n$.
\item[(iii)] For all $i=1,\dots, n$, ${Z}_{\theta_i}$ is a proper subset
of $\mathbb{T}$, where ${Z}_{\theta_i}$ consists of all
points $\lambda$ on $\mathbb{T}$ for which $\theta_i$ cannot be
continued analytically from $\dsc$ to $\lambda$.
\end{itemize}
\endcor

Now we turn to the case of the Bergman module $L^2_a(\dsc^n)$.
For $n=1$, boundary representations corresponding to a quotient module of
$L^2_a(\dsc)$ are studied in ~\cite{WH}.
For a submodule $\cls$ of $L^2_a(\dsc)$, set
\[
 Z_{*}(\cls):=\bigcap\limits_{f\in\cls} Z_*(f),
 \]
 where
\[
Z_*(f)=\big\{\lambda\in\dsc: f(\lambda)=0\big\}\cup\big\{\lambda\in
\mathbb{T}: \liminf\limits_{z\in\dsc,z\to\lambda}|f(z)|=0\big\}.
\]
It is easy to see that for a finite dimensional quotient module
$\clq$ of $L^2_a(\dsc)$, the identity representation of $C^*(\clq)$
is always a boundary representation relative to $B(\clq)$. On the
other hand, for an infinite dimensional  $\clq$, the identity
representation of $C^*(\clq)$ is a boundary representation relative
to $B(\clq)$ if and only if $\mbox{dim}(\cls\ominus z\cls)=1$ and
$Z_*(\cls)$ is a proper subset of $\mathbb{T}$, where
$\cls=L^2_a(\dsc)\ominus\clq$ is the corresponding submodule (see
~\cite[Theorem 1.2]{WH}). Using this and
Theorem~\ref{thm-analhm-bdyrep}, we have the following result.

\begincor
\label{Bergman}
 Let $\clq=\clq_1\ot\cdots\ot\clq_n$ be a doubly commuting quotient module
 of $L^2_a(\dsc^n)$. Then the following are equivalent.
 \begin{itemize}
\item[(i)] The identity representation of
${C}^*(\clq)$ is a boundary representation relative to $B(\clq)$.
\item[(ii)] The identity representation of ${C}^*(\clq_{i})$ is a boundary
representation relative to $B(\clq_{i})$ for all $i=1,\dots,
n$.
\item[(iii)]
If $\clq_i$ $(1\le i\le n)$ is infinite dimensional
then $\mbox{dim}(\cls_i\ominus z\cls_i)=1$
and $Z_*(\cls_i)$ is a proper subset of $\mathbb{T}$,
where $\cls_i=L^2_a(\dsc)\ominus\clq_i$ is the corresponding
submodule.
\end{itemize}
\endcor

\section{Boundary representations for homogeneous quotient
modules}\label{Sect6}

The purpose of this section is to investigate boundary
representations for homogeneous quotient modules of $H^2(\dsc^2)$.
We begin with a lemma which is a standard application
of Arveson's theory on boundary representations
\cite{Arv1, Arv2}. For generality, we prove it for quotient modules
of $H^2(\dsc^n)$.
\textsf{For an  essentially normal quotient module $\clq $ of $H^2(\dsc^n)$,
we denote by $\sigma_{e}(\clq)$ the essential joint spectrum of
$(C_{z_1},\dots,C_{z_n})$}.

\begin{Lemma}\label{c spectral set}
Let $\clq$ be an essentially normal quotient module of
$H^2(\dsc^n)$.
\begin{itemize}
\item[(a)] If there exists a matrix-valued polynomial $p$ such that
\[\norm{p(C_{z_1},\dots,C_{z_n})}> \norm{p}_{\sigma_{e}(\clq)}^{\infty}:=\sup_{\z\in \sigma_{e}(\clq)}\|p(\z)\|,\]
then the identity representation of $C^*(\clq)$ is a boundary
representation relative to $\mathcal{B}(\clq)$.
\item[(b)]
If the commuting tuple $(C_{z_1},\dots,C_{z_n})$ has a normal
dilation on $\sigma_{e}(\clq)$, then the identity representation of
$C^*(\clq)$ is not a boundary representation relative to
$\mathcal{B}(\clq)$.
\end{itemize}
\end{Lemma}

\begin{proof}
a) Since $\clq$ is essentially normal and $C^*(\clq)$ is irreducible, we have that
$K(\clq)\subseteq C^*(\clq)$ and the following
extension
\[
0\longrightarrow K(\clq)\hookrightarrow C^*(\clq)\longrightarrow
C(\sigma_{e}(\clq))\longrightarrow 0.
\]
If there exists a matrix-valued polynomial $p$ such that
\[
\norm{p(C_{z_1},\dots,C_{z_n})}>
\norm{p}_{\sigma_{e}(\clq)}^{\infty},
\]
then the restriction of the canonical contractive homomorphism
\[
q: C^*(\clq)\to C^*(\clq)/K(\clq)\cong C(\sigma_{e}(\clq))
\]
to $\mathcal{B}(\clq)$ is not a complete isometry. The desired
conclusion now follows from the Arveson's boundary theorem \cite[Theorem 2.1.1]{Arv2}.

(b) The existence of a normal dilation implies that the above
completely contractive map $q$ restricted to the linear span of
$\mathcal{B}(\clq)\cup \mathcal{B}(\clq)^*$ is a complete isometry.
Then the conclusion again follows from the Arveson's boundary theorem \cite[Theorem 2.1.1]{Arv2}.
\end{proof}

The following basic property of two variable homogeneous polynomials
will be used subsequently. Given a homogeneous polynomial $p
\in \cplx[z_1,z_2]$ there exist homogeneous polynomials $p_1, p_2
\in \cplx[z_1,z_2]$, unique up to a scalar multiple of modulus one,
such that
\[
 p=p_1p_2,
\]
and
\[
Z(p_1)\cap \bdy\dsc^2\subset \mathbb{T}^2 \quad\text{and}\quad
Z(p_2)\cap \bdy\dsc^2\subset (\dsc\times\mathbb{T})\cup (\mathbb{T}\times \dsc).
\]
Let $pH^2(\dsc^2)$ denote the submodule of $H^2(\mathbb{D}^2)$ generated by
$p$. Suppose that $\clq_p$ is the corresponding quotient module of
$H^2(\mathbb{D}^2)$, that is,
\[
\clq_{p} =H^2(\dsc^2)\ominus p H^2(\dsc^2).
\]
The following characterization of essential normality of $\clq_p$ is
due to Guo and Wang \cite[Theorem 1.1]{GP}.

\bt[Guo \& Wang, ~\cite{GP}] \label{homogenous} Let $p$ be a non-zero homogeneous polynomial
in $\cplx[z_1,z_2]$, and $p=p_1p_2$ be the factorization of $p$ as
above. Then the quotient module $\clq_{p}$ is essentially normal if
and only if $p_2$ has one of the following forms:
\begin{itemize}
\item[(i)]
$p_2 \equiv c$ with $c\neq 0$,
\item[(ii)]
 $p_2=\alpha z_1+\beta z_2$ with $|\alpha|\neq|\beta|$,
 \item[(iii)]
 $p_2=c(z_1-\alpha z_2)(z_2-\beta z_1)$ with $|\alpha|<1, |\beta|<1$ and
 $c\neq 0$.
 \end{itemize}
\et

The following result in \cite{GP} gives a description of the essential joint spectrum of the
above type of quotient modules. For a proof we refer the reader to ~\cite[Theorem 6.2]{GP}.

\begin{Lemma}[Guo \& Wang, ~\cite{GP}]
\label{newlabel}
 Let $p$ be a homogeneous polynomial. Then
 \[\sigma_{e}(\clq_{p})=Z(p)\cap \bdy\dsc^2.\]
\end{Lemma}
For our present purposes, however, we need only the fact that
$\sigma_{e}(\clq_{p})\subset Z(p)\cap \bdy\dsc^2$.

We now state our main result of this section. This gives a
partial characterization of boundary representations for the
class of essentially normal homogeneous quotient modules of
$H^2(\dsc^2)$.
\smallskip

\bt\label{thm6} Let $p \in \mathbb{C}[z_1, z_2]$ be a homogeneous
polynomial. Suppose that $\clq_{p}$ is an essentially normal
quotient module of $H^2(\dsc^2)$. Then the identity representation
of $C^*(\clq_{p})$ is a boundary representation relative to
$B(\clq_{p})$
if $p$ is not of the following form:\\
\textup{(i)} $p= c(z_1^m-\alpha z_2^m)$ for some $m\in\Nat$,
 $c\neq 0$ and $|\alpha|=1$,\\
 \textup{(ii)} $p=\alpha z_1+ \beta z_2$ with $|\alpha|\neq |\beta|$.

 Furthermore, if $p$ is either as in \textup{(i)} with $m=1$ or as in \textup{(ii)} then the identity representation of $C^*(\clq_p)$
 is not a boundary representation.
\et

Our proof of Theorem~\ref{thm6} on boundary representations
for homogeneous quotient modules is based on the following two
special cases. A couple of lemmas below describes these.

\begin{Lemma}
\label{multiplicity}
 Let $p=c\prod_{i=1}^m(z_1-\alpha_iz_2)^{n_i}$ be a homogeneous polynomial with
 $c\neq 0$ and $\alpha_i$'s are distinct scalars of modulus one.
 Assume further that $n_i>1$ for some $i=1,\dots,m$. Then the identity
 representation of $C^*(\clq_{p})$ is a boundary representation relative to
 $B(\clq_{p})$.
\end{Lemma}
\begin{proof}
 Without loss of any generality assume that $n_1>1$. Set
\[
q(z_1,z_2):=(z_1-\alpha_1z_2)\prod_{i=2}^m (z_1-\alpha_i z_2)^{n_i}.
\]
Then, $q(C_{z_1},C_{z_2})$ is a non-zero operator and
$\norm{q}^{\infty}_{Z(p)}=0$. Hence, by Lemma~\ref{newlabel} and  part (a) of
Lemma~\ref{c spectral set}, the identity representation of
$C^*(\clq_{p})$ is a boundary representation relative to
$B(\clq_{p})$.
\end{proof}

\begin{Lemma}
\label{linear} Let $p=c(z_1-\alpha z_2)$, for some $\alpha\in\cplx$
and $c\neq 0$. Then, the identity representation of
$C^*(\clq_{p})$ is not a boundary representation relative to
$B(\clq_{p})$.
\end{Lemma}
\begin{proof}
Let $\alpha=0$. Then $\clq_{p}$ is unitarily equivalent to
$H^2(\dsc)$ and, hence, the conclusion follows easily. Now let
$|\alpha |=1$. In this case $\clq_{p}$ is unitary equivalent to the
Bergman space over the unit disc and therefore the result follows.
Finally,  let
$\alpha\neq 0$ and $|\alpha |\neq 1$. Then
\[
\frac{1}{\alpha}C_{z_1}= C_{z_2},
\]
and hence $C^*(\clq_{p})$ is generated by $C_{z_1}$. Now assume that
$|\alpha|>1$ (the $|\alpha|<1$ case is similar). Set $\beta =
\frac{\alpha}{|\alpha|^2}$ and
\[
c_n:=(\sum_{m=0}^n |\beta|^{2m})^{1/2} \quad \quad (n\in\Nat).
\]
It follows that the sequence of homogeneous polynomials $\{p_n\}$ is
an orthonormal basis for $\clq_{p}$ \cite{DP}, where
\[
p_n(z)=\frac{1}{c_n}\sum_{m=0}^n z_1^m(\beta z_2)^{n-m}\quad (n\in\Nat).
\]
Moreover (again see \cite{DP}), for all $n\ge 0$,
\[
\begin{split}
C_{z_1}(p_n)& = P_{\clq_{p}} (\frac{1}{c_n} \sum_{m=0}^n
z_1^{m+1}(\beta z_2)^{n-m})
\\
&= \frac{1}{c_n} \langle p_{n+1}, \sum_{m=0}^n z_1^{m+1}(\beta z_2)^{n-m} \rangle\ p_{n+1}
\\
&=\frac{c_n}{c_{n+1}} p_{n+1}.
\end{split}
\]
Thus, by Theorem~\ref{homogenous}, $C_{z_1}$ is an essentially normal weighted shift with
weights $\{\frac{c_n}{c_{n+1}}\}_{n\ge 0}$. Finally, since $\lim
\sup_n \frac{c_n}{c_{n+1}}= \sup_n \frac{c_n}{c_{n+1}}$, the result
follows from \cite[Corollary 2]{Arv2}.
\end{proof}

We are now in a position to give a proof of Theorem~\ref{thm6}.

\begin{proof}[Proof of Theorem~\ref{thm6}]

We first note that, since $\clq_{p}$ is
essentially normal, $p$ can be represented as in Theorem~\ref{homogenous}. Now by
Lemma~\ref{multiplicity}, is it enough to consider the case
$p=p_1p_2$, where
\[
p_1(z)=\prod_{i=1}^m(z_1-\alpha_iz_2),
\]
$\alpha_i,\;i=1,\dots, m$, are all distinct scalars of modulus one, and
$p_2$ is as in Theorem~\ref{homogenous}. In view of the forms of
$p_2$ in Theorem~\ref{homogenous}, we next consider the following
four cases.
\smallskip

{\sf Case I:} Let $p_2=c$, $p_1=\prod_{i=1}^m(z_1-\alpha_iz_2)$,
$\alpha_i,\;i=1,\dots, m$, are all distinct scalars of modulus one, and
that $p=p_1p_2$ is not of the form $c(z_1^m-\alpha z_2^m)$.

\noindent In this case we have $m>1$. Set
\[
q(z) = \frac{1}{z_1} (\prod_{i=1}^m(z_1-\alpha_i
z_2)-(-1)^m(\prod_{i=1}^m\alpha_i) z_2^m).
\]
Then
\begin{equation}\label{polynomial}
\begin{split}
q(z) = \sum_{k=0}^{m-1}(-1)^k\left(\sum_{1\leq i_1<\dots < i_k\leq m}
\alpha_{i_1}\dots\alpha_{i_k}\right)z_1^{m-k-1}z_2^k.
\end{split}
\end{equation}
A simple calculation shows that
\[
\norm{q}^{\infty}_{Z(p)\cap \bdy\dsc^2}=1.
\]
On the other hand, note that $1\in \clq_p$ as $p$ vanishes at $0$ and $q\in \clq_p$ as the degree of $q$ is strictly less than that of the homogeneous polynomial $p$. Then we have
\[
\norm{q(C_{z_1},C_{z_2})}\ge \norm{q}_{H^2(\dsc^2)}= \sqrt{1+
\sum_{k=1}^{m-1}\left|\sum_{1\leq i_1< \dots< i_k\leq m}
\alpha_{i_1}\dots\alpha_{i_k}\right|^2}.
\]
Now if
\[
\sum_{k=1}^{m-1}\left|\sum_{1\leq i_1<\dots< i_k\leq m}
\alpha_{i_1}\dots\alpha_{i_k}\right|^2=0,
\]
then \eqref{polynomial} yields that
\[
p_1(z) = \prod_{i=1}^m(z_1-\alpha_i z_2) = z_1^{m}-\alpha z_2^m,
\]
for some $\alpha$ of modulus one which is an obvious contradiction.
Thus,
\[
\norm{q(C_{z_1},C_{z_2})}>1,
\]
 and therefore, by Lemma~\ref{newlabel} and Lemma~\ref{c
spectral set}, the identity representation is a boundary
representation in this case.
\smallskip

{\sf Case II:} Let $p_2(z)=(z_1-\gamma z_2)$, $|\gamma|\neq 1$,
$p_1(z)=\prod_{i=1}^m(z_1-\alpha_iz_2)$, where $\alpha_i,\; i=1,\dots,m$, are
distinct scalars of modulus one. Also, without loss of generality,
we may assume that $|\gamma|<1$. Otherwise, by interchanging the
role of $z_1$ and $z_2$, we can consider that
$p_2(z)=(z_2-\frac{1}{\gamma}z_1)$ and
$p_1(z)=\prod_{i=1}^m(z_2-\frac{1}{\alpha_i}z_1)$. As in the previous
case, set
\[
q(z) = \frac{1}{z_1} (\prod_{i=1}^{m+1}(z_1-\alpha_i z_2)-(-1)^{m+1}
\prod_{i=1}^{m+1}\alpha_i z_2^{m+1}),
\]
where $\alpha_{m+1}=\gamma$. A similar computation shows that
\[
\norm{q}^{\infty}_{Z(p)\cap\bdy\dsc^2}=1,
\]
and
\[\norm{q(C_{z_1},C_{z_2})}\ge \sqrt{1+
\sum_{k=1}^{m}\left|\sum_{1\leq i_1<\dots< i_k\leq m+1}
\alpha_{i_1}\dots\alpha_{i_k}\right|^2}.
\]
Therefore, if
\[
\sum_{k=1}^{m}\left|\sum_{1\leq i_1<\dots< i_k\leq m+1}
\alpha_{i_1}\dots\alpha_{i_k}\right|^2=0,
\]
then
\[
\prod_{i=1}^{m+1}(z_1-\alpha_i z_2)= z_1^{m+1}-\alpha z_2^{m+1},
\]
for some scalar $\alpha$ with $|\alpha|\neq 1$.
Consequently

\[
|\alpha_1| = \dots=|\alpha_{m+1}|=|\alpha|^{1/(m+1)}\neq 1,
\]
which is a contradiction. Therefore $\norm{q(C_{z_1},C_{z_2})}>1$,
and the conclusion again follows from Lemma~\ref{newlabel} and Lemma~\ref{c spectral set}.
\smallskip

{\sf Case III:} Let $p=p_2=(z_1-\gamma_1 z_2)(z_2-\gamma_2 z_1)$,
$|\gamma_1|< 1$ and $|\gamma_2|<1$. We further divide it into two sub-cases.
First assume that $\gamma_2\neq 0$.  In this sub-case, without any loss of generality, we take $p=(z_1-\gamma_1z_2)(z_1-\gamma_2z_2)$ with
$|\gamma_1|<1$ and $|\gamma_2|>1$.
For each $\epsilon>0$, set
\[
q_{\epsilon}:=(z_1-\epsilon \gamma_2 z_2) \in \mathbb{C}[z_1, z_2],
\]
and set
\[
V_1 =\{(\gamma_1z_2,z_2)\in\bdy\dsc^2: |z_2|=1 \} \text{ and } V_2
=\{(\gamma_2z_2,z_2)\in\bdy\dsc^2: |\gamma_2z_2|=1\}.
\]
Note that
\[
Z(p)\cap\bdy\dsc^2=V_1\cup V_2,
\]
and
\[
\norm{q_{\epsilon}}^{\infty}_{V_1}= |\gamma_1-\epsilon\gamma_2|, \quad
\norm{q_{\epsilon}}^{\infty}_{V_2}=(1-\epsilon).
\]
Therefore, for a sufficiently small $0<\epsilon<1$, we obtain
\[
\norm{q_{\epsilon}}^{\infty}_{Z(p)\cap\bdy\dsc^2}<1.
\]
On the other hand, for any $0<\epsilon<1$,
\[
\norm{q_{\epsilon}(C_{z_1},C_{z_2})}\ge
\sqrt{1+|\epsilon\gamma_2|^2}>1,
\]
and the conclusion again follows from Lemma~\ref{c spectral set}.

If $\gamma_2=0$, then $p=z_2(z_1-\gamma_1z_2)$. For each $\epsilon >0$,
consider
\[
q_{\epsilon}=z_1-\epsilon z_2,
\]
and set
\[
V_1 =\{(\gamma_1z_2,z_2)\in\bdy\dsc^2: |z_2|=1 \} \text{ and } V_2
=\{(z_1,0)\in\bdy\dsc^2: |z_1|=1\}.
\]
Then as before, one can check that
\[
Z(p)\cap\bdy\dsc^2=V_1\cup V_2,\
\|q_{\epsilon}\|^{\infty}_{V_1}=|\gamma_1-\epsilon |\ \text{ and }\
\|q_{\epsilon}\|^{\infty}_{V_2}=1.
\]
Thus, for a sufficiently small $0<\epsilon<1$, we have
\[
\norm{q_{\epsilon}}^{\infty}_{Z(p)\cap\bdy\dsc^2}=1.
\]
On the other hand, for any $0<\epsilon<1$,
\[
\norm{q_{\epsilon}(C_{z_1},C_{z_2})}\ge
\sqrt{1+|\epsilon |^2}>1,
\]
and therefore the conclusion follows from Lemma~\ref{c spectral set}.
\smallskip

{\sf Case IV:}  Let $p_2(z)=(z_1-\gamma_1 z_2)(z_2-\gamma_2 z_1)$,
$|\gamma_1|< 1$, $|\gamma_2|<1$,
$p_1(z)=\prod_{i=1}^m(z_1-\alpha_iz_2)$, and $\alpha_i,\; i=1,\dots,m$, are
distinct scalars of modulus one. Set
\[
V_{\gamma_1} =\{(\gamma_1z_2,z_2)\in\bdy\dsc^2: |z_2|=1 \},\quad
V_{\gamma_2} =\{(z_1, \gamma_2z_1)\in\bdy\dsc^2: |z_1|=1\},
\]
and
\[
V_{\alpha_i}:=\{(\alpha_iz_2,z_2)\in\bdy\dsc^2: |z_2|=1 \}, \quad
i=1,\dots, m.
\]
We now consider, for $\epsilon>0$,
\[
 q_{\epsilon}=z_2(z_1^m+\epsilon q') \in \mathbb{C}[z_1, z_2],
\]
where
\[
q'(z)=\prod_{i=1}^m(z_1-\alpha_iz_2)-z_1^m=\sum_{k=1}^m (-1)^k\left(\sum_{1\leq i_1<\dots< i_k\leq m}
\alpha_{i_1}\dots\alpha_{i_k}\right)z_1^{m-k}z_2^{k}.
\]
Then again by a simple calculation we get
\[
\norm{q_{\epsilon}}^{\infty}_{V_{\gamma_1}}\le |\gamma_1^m|+\epsilon
M, \quad \norm{q_{\epsilon}}^{\infty}_{V_{\gamma_2}}\le
|\gamma_2|(1+\epsilon M),
\]
and
\[
\norm{q_{\epsilon}}^{\infty}_{V_{\alpha_i}}=|(1-\epsilon)|, \quad
i=1,\dots, m,
\]
where
\[
M = \max\{|q'(\gamma_1,1)|,|q'(1, \gamma_2)|\}.
\]
We now choose $0<\epsilon<1$ so that
$\norm{q_{\epsilon}}^{\infty}_{Z(p)\cap\bdy\dsc^2}<1$. On the other
hand, since $\norm{q_{\epsilon}(C_{z_1},C_{z_2})}\ge
\norm{q_{\epsilon}}>1$, the conclusion follows immediately.
Thus we have the first part of the theorem.

The last part follows from Lemma~\ref{linear}.
This completes the proof of the theorem.
\end{proof}

Note that Theorem \ref{thm6} completely describes the issue of
boundary representations for all essentially normal homogeneous
quotient modules of $H^2(\dsc^2)$, except when $p=(z_1^m-\alpha
z_2^m)$, $|\alpha |=1$ and $m\ge  2$.
We conclude this paper with the following question: Let $|\alpha
|=1$, $m \ge 2$ and let $p=z_1^m-\alpha z_2^m$. Is the identity
representation of $C^*(\clq_{p})$ a boundary representation?

\bigskip

\noindent\textbf{Acknowledgment:} We are very grateful to the
referee for a careful reading of the manuscript and valuable
suggestions and thoughtful comments. The first two authors are
grateful to Indian Statistical Institute, Bangalore Center for warm
hospitality. The first named author's research work is supported by
DST-INSPIRE Faculty Fellowship No. DST/INSPIRE/04/2015/001094.
The second named author is supported by an INSPIRE faculty fellowship (IFA-MA-02)
funded by DST. The third author is supported in part by NBHM
(National Board of Higher Mathematics, India) grant
NBHM/R.P.64/2014.


\end{document}